\documentclass[10pt,a4paper]{amsart}
\usepackage[latin1]{inputenc}
\usepackage{amsmath,enumerate,amsthm,graphicx,color,verbatim}
\usepackage{amsfonts,amscd}
\usepackage{latexsym}
\usepackage{amssymb}
\newtheorem{thm}{Theorem}
\newtheorem{defn}{Definition}

\newcommand{\Z}{\mathbb{Z}}
\newcommand{\D}{\mathbb{D}}

\newcommand{\norm}{\vert\vert}

\begin{document}
\title[OPUC with quasiperiodic Verblunsky Coefficients]{Orthogonal Polynomials on the Unit Circle with quasiperiodic Verblunsky Coefficients have generic purely singular continuous spectrum}

\author[D.\ Ong]{Darren C. Ong}

\address{Department of Mathematics, Rice University, Houston, TX~77005, USA}

\email{darren.ong@rice.edu}

\urladdr{http://math.rice.edu/$\sim$do3}

\thanks{Supported in part by NSF grant
DMS-1067988.}
\maketitle
\begin{abstract}
As an application of the Gordon lemma for orthogonal polynomials on the unit circle, we prove that for a generic set of quasiperiodic Verblunsky coefficients the corresponding two-sided CMV operator has purely singular continuous spectrum. We use a similar argument to that of the Boshernitzan-Damanik result that establishes the corresponding theorem for the discrete Schr\"odinger operator.
\end{abstract}
\begin{paragraph}{\textbf{Keywords}} Spectral Theory, Orthogonal Polynomials on the Unit Circle, Almost Periodicity.
\end{paragraph}
\begin{section}{Introduction}
We concern ourselves primarily with the one-dimensional discrete Schr\"odinger operator with quasiperiodic potential. The discrete Schr\"odinger operator is a $\ell^2(\mathbb Z)$ operator $H=\Delta+V_n$ where $\Delta$ is the discrete Laplacian and $V_n$ is known as the potential sequence. The unitary analogue of the discrete Schr\"odinger operator is the CMV operator, the $\ell^2(\mathbb Z)$ operator $\mathcal E$ given by
\begin{equation*}\label{CMVmatrix}
\left(
\begin{array}{ccccccc}
\ldots&\ldots& \ldots&\ldots&\ldots&\ldots&\ldots\\
\ldots&-\overline{\alpha(0)}\alpha(-1)&\overline{ \alpha(1)}\rho(0)&\rho(1)\rho(0)&0&0&\ldots\\
\ldots&-\rho(0)\alpha(-1)&- \overline{\alpha(1)}\alpha(0)&-\rho(1)\alpha(0)&0&0&\ldots\\
\ldots&0& \overline{\alpha(2)}\rho(1)&-\overline{\alpha(2)}\alpha(1)&\overline{\alpha(3)}\rho(2)&\rho(3)\rho(2)&\ldots\\
\ldots&0& \rho(2)\rho(1)&-\rho(2)\alpha(1)&-\overline{\alpha(3)}\alpha(2)&-\rho(3)\alpha(2)&\ldots\\
\ldots&0& 0&0&\overline{\alpha(4)}\rho(3)&-\overline{\alpha(4)}\alpha(3)&\ldots\\
\ldots&\ldots& \ldots&\ldots&\ldots&\ldots&\ldots\\
\end{array}
\right),
\end{equation*}
where $\{\alpha(n)\}_{n\in\Z}$ is a sequence in $\mathbb D$, the open unit disk, and $\rho(n)=(1-\vert\alpha(n)\vert^2)^{1/2}$. The $\alpha(n)$ are known as Verblunsky coefficients. The one-sided version of this operator is essential in the study of orthogonal polynomials corresponding to probability measures on the unit circle. The Verblunsky coefficients are the recurrence coefficients of those polynomials, and the spectral measure of the one-sided operator is exactly the probabilty measure used to generate the polynomials. This is analogous to the way that the one-sided discrete Schr\"odinger operator is related to the study of orthogonal polynomials corresponding to probability measures on the real line.

It is thus natural to ask if the results pertaining to the discrete Schr\"odinger operator with quasiperiodic potential also hold in the context of CMV operators with quasiperiodic Verblunsky coefficients. In this note, we observe that some facts pertaining to generic spectral decomposition properties of the quasiperiodic Schr\"odinger operator apply to the CMV case as well. Specifically, we assert that the result of \cite{Damanik-Boshernitzan} has a CMV analogue. 

Let us first explain the dynamical setting that underlies the results of this note. We define $\Omega$ as a compact metric space, $T:\Omega\to\Omega$ a homeomorphism and $\mu$ a $T$-ergodic Borel probability measure on $\Omega$. We also define $f$, the sampling function to be a continuous function on $\Omega$. For the Schr\"odinger operator, the range of $f$ is $\mathbb R$ and we let $V_n=f(T^n\omega)$, for $\omega\in \Omega$. Similarly, for our unitary setting $f$ is a function in $C(\Omega,\mathbb D)$ and we let $\alpha(n)=f(T^n\omega)$, for $\omega\in \Omega$. We express the corresponding CMV operator as $\mathcal E_\omega$.

An important example of quasiperiodic dynamics is the minimal shift $T\omega=\omega+\mathfrak a$ on a $d$-dimensional torus, $\Omega=\mathbb T^d$. The paper \cite{Damanik-Boshernitzan} proves generic absence of point spectrum of the Schr\"odinger Operator for quasiperiodic rotations on a $d$-dimensional torus and for skew-shifts on $\mathbb T^2$. In this note, we verify that with some modifications the same arguments apply to the CMV context as well:

\begin{thm}\label{thmcont}
For almost every $\omega\in\Omega$ and $T$ either a minimal shift on $\Omega=\mathbb T^d$ or a skew-shift  on $\mathbb T^2$ given by $T(\omega_1,\omega_2)=(\omega_1+2\mathfrak a,\omega_1+\omega_2)$, there exists a dense $G_\delta$ set of sampling functions in $C(\Omega,\D)$ for which the corresponding CMV operator has empty point spectrum.
\end{thm}
Also, we can deduce as an immediate corollary,
\begin{thm}\label{gensing}
For almost every $\omega\in\Omega$ and $T$ either a minimal shift on $\Omega=\mathbb T^d$ or a skew-shift  on $\mathbb T^2$ given by $T(\omega_1,\omega_2)=(\omega_1+2\mathfrak a,\omega_1+\omega_2)$, there exists a dense $G_\delta$ set of sampling functions in $C(\Omega,\D)$ for which the corresponding CMV operator has purely singular continuous spectrum.
\end{thm}
\begin{proof}
Theorem \ref{thmcont} says that we have absence of point spectrum. According to \cite{Zhang}, we have generic singular spectrum for orthogonal polynomials on the unit circle with aperiodic Verblunsky coefficients. The conclusion follows immediately.
\end{proof}
\end{section}
\begin{section}{Definitions and preliminaries}
We begin with some setup. Any reference to (eTRP), (eMRP), (eGRP) refers to the definitions (TRP), (MRP), (GRP), expressed as Definitions 2,3 and 4 in \cite{Damanik-Boshernitzan} with a small modification. More specifically, we replace the repetition property in those definitions with the following even repetition property

\begin{defn}[Even Repetition Property]\label{RP}
 A sequence $\{\omega_k\}_{k\in[0,j]}$ has the even repetition property if for every $\epsilon>0, s>0$ there exists an \underline{even} integer $q\geq2$ such that  $\mathrm{dist}(\omega_k, \omega_{k+q})<\epsilon$ for $k=0,\ldots, \lfloor sq\rfloor$. 
 \end{defn}
 
 More explicitly, $(\Omega, T)$ satisfies the (eTRP) condition if the set of points in $\Omega$ whose forward orbit satisfies the even repetition property is dense in $\Omega$. $(\Omega, T)$ satisfies the (eMRP) property if that set is full measure instead of dense, and it satisfies (eGRP) if all forward orbits satisfy the even repetition property.
 
Our Verblunsky coefficients will be generated by a homeomorphism $T$ of a compact metric space $\Omega$ and a continuous sampling function $f\in C(\Omega, \mathbb D)$ so that $\alpha(n)=f(T^n\omega)$. 

According to \cite{Ong-LP}, there must exist a positive real-valued function $\Gamma(k,q,r)$ with $k,q\in \mathbb Z_+$ so that if $\{\alpha(n)\}$, $\{\tilde \alpha(n)\}$ are two sequences of Verblunsky coefficients with distance at most $r$ from the origin, corresponding to the matrices $A_n, \tilde A_n$ respectively, (where these matrices are defined in Section 2 of \cite{Ong-LP}) and if 
\[\vert\tilde\alpha(n)-\alpha(n)\vert, \vert\tilde\alpha({n+1})-\alpha({n+1})\vert,\vert\tilde\alpha({n+2})-\alpha({n+2})\vert,\]
 are all less than $\Gamma(k,q,r)$, this implies that $\norm A_n-\tilde A_n\norm< k^{-q}$. We then define a Gordon sequence:
\begin{defn}[Gordon sequence]
A two-sided sequence of Verblunsky coefficients $\alpha(n)$ such that there exists a sequence of even positive integers $q_k\to \infty$ so that 
\[
\max_{-q_k+1\leq n\leq q_k+1}\vert \alpha(n)-\alpha(n\pm q_k)\vert\leq \frac{\Gamma(k, q_k,r_k)}{4}.
\]
is a \em Gordon sequence \em.
Here, $r_k$ is a positive real number less than $1$ for which $\alpha(-2q_k+1),\ldots, \alpha(2q_k+1)$ all lie in a disk of radius $r_k$ centered at the origin.

\end{defn}
According to Corollary 1 of \cite{Ong-LP}, CMV operators corresponding to such sequences have no point spectrum.
\end{section}
\begin{section}{Main theorems and proofs}
\begin{thm}\label{TRPproof}
Suppose $(\Omega, T)$ is minimal and satisfies (eTRP). Then there exists a dense $G_\delta$ subset of $\mathcal F$ of $C(\Omega, \mathbb D)$ such that for every $f\in \mathcal F$, there is a residual subset $\Omega_f\subseteq \Omega$ with the property that for every $\omega\in \Omega_f$, the sequence of two-sided Verblunsky coefficients defined above is a Gordon sequence.
\end{thm}
\begin{proof}
The proof is similar to the proof found in \cite{Damanik-Boshernitzan}, with some small modifications.

By assumption, there is a point $\omega\in \Omega$ whose forward orbit has the even repetition property.  For each $k\in \mathbb Z_+$, we consider $\epsilon=1/k$, $s=4$ and the associated $q_k=q(\epsilon,s)$ in Definition \ref{RP}. We then have that $q_k\to\infty$ as $k\to\infty$. Now take an open ball $B_k$ around $\omega$ with radius small enough so that 
 \[\overline{T^n(B_k)}, 1\leq n\leq 5q_k\] are disjoint and, for every $1\leq j\leq q_k$, 
 \[\bigcup _{l=0}^4 T^{j+lq_k} (B_k)\]
 is contained in some ball of radius $5\epsilon$. Define 
 \[\mathcal C_k=\{ f\in C(\Omega, \mathbb D):f \text{ is constant on each set } \bigcup_{l=0}^4 T^{j+lq_k} (B_k), 1\leq j\leq q_k\},\]
 also define
 \[\mathcal F_k=\left\{f\in C(\Omega, \mathbb D): \exists \tilde f\in \mathcal C_k, \text{ such that } \norm f-\tilde f\norm <\frac{1}{2}\left(\frac{\Gamma(k,q_k,\norm f\norm)}{4}\right) \right\}. \]
 Note that $\mathcal F_k$ is an open neighborhood of $\mathcal C_k$, and hence for each $m$
 \[\bigcup_{k\geq m} \mathcal F_k\]
is an open and dense subset of $C(\Omega, \mathbb D)$. This follows since every $f\in C(\Omega,\mathbb D)$ is uniformly continuous and the diameter of the set $\bigcup_{l=0}^4 T^{j+lq_k} (B_k)$ goes to zero, uniformly in $j$ as $k\to\infty$. Thus 
\[\mathcal F=\bigcap_{m\geq 1}\bigcup _{k\geq m} \mathcal F_k\]
is a dense $G_\delta$ subset of $C(\Omega, \mathbb D)$.

Consider some $f\in \mathcal F$. Then $f\in \mathcal F_{k_l}$ for some sequence $k_l\to\infty$. Observe that for every $m\geq 1$,
\[
\bigcup _{l\geq m} \bigcup_{j=1}^{q_{k_l}} T^{j+q_{k_l}}(B_{k_l})
\]
is an open and dense subset of $\Omega$ since $T$ is minimal and $q_{k_l}\to \infty$. Thus 
\[\Omega_f=\bigcap_{m\geq 1}\bigcup_{l\geq m} \bigcup_{j=1}^{q_{k_l}}T^{j+q_{k_l}}(B_{k_l})\]
is a dense $G_\delta$ subset of $\Omega$.

Given $\omega\in\Omega_f$, $\omega $ belongs to $\bigcup _{j=1}^{q_{k_l}} T^{j+q_{k_l}}(B_{k_l})$ for infinitely many $l$. For each such $l$, we have by construction that 
\[\max_{1\leq j\leq q_{k_l}}\vert f(T^{j+q_{k_l}} \omega)-f(T^{j+2q_{k_l}}\omega)\vert <\frac{\Gamma(k, q_k, \norm f\norm )}{4},\]
\[\max_{1\leq j\leq q_{k_l}}\vert f(T^j \omega)-f(T^{j+q_{k_l}}\omega)\vert <\frac{\Gamma(k, q_k, \norm f\norm )}{4},\]
\[\max_{1\leq j\leq q_{k_l}}\vert f(T^{j-q_{k_l}} \omega)-f(T^{j}\omega)\vert <\frac{\Gamma(k, q_k, \norm f\norm )}{4}\]
and 
\[\max_{1\leq j\leq q_{k_l}}\vert f(T^{j-2q_{k_l}} \omega)-f(T^{j-q_{k_l}}\omega)\vert <\frac{\Gamma(k, q_k, \norm f\norm )}{4}.\]
Thus $\alpha(n)=f(T^n\omega)$ is Gordon.
\end{proof}
\begin{thm}\label{MRPproof}
Suppose that $(\Omega, T,\mu)$ satisfies (eMRP). Then there exists a residual subset $\mathcal F$ of $C(\Omega,\mathbb D)$ such that for every $f\in \mathcal F$, there exists a subset $\Omega_f\subseteq \Omega$ of full $\mu$ measure with the property that for every $\omega\in\Omega_f$,  the corresponding sequence of two-sided Verblunsky coefficients is a Gordon sequence.
\end{thm}
The proof of this theorem is almost exactly the same as that of Theorem 3 of \cite{Damanik-Boshernitzan}, except that we define

 \[\mathcal F_i=\left\{f\in C(\Omega, \mathbb D): \exists \tilde f\in F_i, \text{ such that } \norm f-\tilde f\norm <\frac{1}{2}\left(\frac{\Gamma(k,q_k,\norm f\norm)}{4}\right) \right\}, \]
 
 and make some other minor adjustments.

\begin{thm}\label{minimalshift}
Every minimal shift $T\omega=\omega+\mathfrak a$ on the torus $\mathbb T^d$ satisfies (eGRP).
\end{thm}
\begin{proof}
Note that the assumption that the condition that $(\mathbb T^d,T)$ satisfies (eGRP) is the same thing as saying that $(\mathbb T^d, T')$ defined by $T'\omega=\omega+2\mathfrak a$ satisfies (GRP). But this is true by Theorem 3 of \cite{Damanik-Boshernitzan}.
\end{proof}

\begin{defn}
$\mathfrak a\in \mathbb T$ is called \em badly approximable \em if there exists a constant $c>0$ such that $\left<\mathfrak a q\right>>c/q$ for every $q\in \mathbb Z\setminus \{0\}$, where $\left<\ldots\right>$ refers to distance from $\mathbb Z$.
\end{defn}

\begin{thm}\label{skewshift}
For a minimal skew-shift $T(\omega_1,\omega_2)=(\omega_1+2\mathfrak a,\omega_1+\omega_2)$ on the torus $\mathbb T^2$, the following are equivalent.
\begin{enumerate}[(i)]
\item $\mathfrak a $ is not badly approximable
\item $(\Omega, T)$ satisfies (eGRP)
\item $(\Omega, T,\mathrm{Leb})$ satisfies (eMRP)
\item $(\Omega, T)$ satisfies (eTRP)
\end{enumerate} 
\end{thm}
\begin{proof}
Note that since (eTRP) is a stronger condition than (TRP), and (eGRP)$\implies$ (eMRP)$\implies$(eTRP), and since Theorem 4 in \cite{Damanik-Boshernitzan} asserts that (TRP)$\implies$ $(i)$, we only need to prove $(i)\to (ii)$.

Assume that $\mathfrak a$ is not badly approximable. This means that there is some positive integer sequence $q_k\to\infty$ such that 
\[\lim_{k\to\infty} q_k\left<\mathfrak a q_k\right>=0.\]
Furthermore, if we double all the $q_k$s, clearly this limit still equals zero. Thus we may assume all the $q_k$s are even.

Iterating the skew shift $n$ times, we find 
\[ T^n(\omega_1,\omega_2)=(\omega_1+2n\mathfrak a, \omega_2+n\omega_1+n(n-1)\mathfrak a),\]

Therefore, 
\begin{equation}\label{iterate}
T^{n+q}(\omega_1,\omega_2)-T^n(\omega_1,\omega_2)=(2q\mathfrak a, q\omega_1+q^2\mathfrak a+2nq\mathfrak a-q\mathfrak a).
\end{equation}

Let $(\omega_1,\omega_2)\in\mathbb T^2$, $\epsilon>0$ and $r>0$ be given. We will construct an even sequence $\tilde q_k\to\infty$ so that for $1\leq n\leq r\tilde q_k$,
\[(2\tilde q_k\mathfrak a, \tilde q_k\omega_1+\tilde q_k^2\mathfrak a+2n \tilde q_k\mathfrak a-\tilde q_k\mathfrak a)\]
is of size $O(\epsilon)$. Each $\tilde q_k$ will be of the form $m_k q_k$ for some $m_k\in \{1,2,\ldots, \lfloor \epsilon^{-1}\rfloor+1\}$. We can see that every term of Equation \ref{iterate} except $\tilde q_k\omega_1$ goes to zero as $k\to\infty$, regardless of the choice of $m_k$, and hence is less than $\epsilon$ for $k$ large enough. To treat the remaining term, we can just choose $m_k$ in the specified $\epsilon$-dependent range so that $\tilde q_k\omega_1=m_k(q_k\omega_1)$ is of size less than $\epsilon$ as well. Consequently, the orbit of $(\omega_1,\omega_2)$ has the repetition property. Since $(\omega_1,\omega_2)$ was arbitrary, (eGRP) holds.
\end{proof}
\begin{proof}[Proof of Theorem \ref{thmcont}]
By Theorem \ref{minimalshift}, we know that minimal shifts satisfty (eGRP) and hence (eMRP) and (eTRP). By Theorem \ref{skewshift} we know that skew shifts with  $\mathfrak a$ not badly approximable satisfy (eMRP). Note that the set of badly approximable $\mathfrak a$'s has zero Lebesgue measure (\cite{Khintchin}, Theorem 29 on p.60). We then apply Theorems \ref{TRPproof} and \ref{MRPproof} to conclude our proof.
\end{proof}
\end{section}

\bibliographystyle{alpha}   
\bibliography{mybib}
\end{document}